\documentclass{amsart}

\usepackage{amssymb}
\usepackage{amsmath}
\usepackage{amsfonts}
\usepackage{graphicx}

\newtheorem{definition}{Definition}
\newtheorem{lemma}[definition]{Lemma}
\newtheorem{theorem}[definition]{Theorem}
\newtheorem{proposition}[definition]{Proposition}

\newtheorem{remark}[definition]{Remark}

\newtheorem{ep}[definition]{Example}

\usepackage[all]{xy}

\newcommand{\ZZ}{\ensuremath{\mathbb Z}}

\newcommand{\RR}{\ensuremath{\mathbb R}}

\newcommand{\cO}{\mathcal{O}}
\newcommand{\cB}{\mathcal{B}}
\newcommand{\cN}{\mathcal{N}}
\newcommand{\cL}{\mathcal{L}}
\newcommand{\cD}{\mathcal{D}}

\newcommand{\cI}{\mathcal{I}}
\newcommand{\cM}{\mathcal{M}}
\newcommand{\cC}{\mathcal{C}}
\newcommand{\cA}{\mathcal{A}}
\newcommand{\cP}{\mathcal{P}}
\newcommand{\cS}{\mathcal{S}}

\newcommand{\un}{\underline}

\begin{document}

\title{Graded geometry and Poisson reduction}

 

\author{A.S. Cattaneo}\address{{Institut f\"ur Mathematik, Universit\"at Z\"urich-Irchel, Winterthurerstr. 190, CH-8057 Z\"urich, Switzerland} }\email{alberto.cattaneo@math.uzh.ch}
\author{M. Zambon}\address{{Centre de Recerca Matematica, Apartat de correus 50, 08193
Bellaterra, Spain} }
 
\begin{abstract}
A result  of \cite{CZ} extends the Marsden-Ratiu reduction theorem \cite{MR} in Poisson geometry, and is  proven by means of graded geometry. In this  note we  provide the background material about graded geometry necessary for the proof in \cite{CZ}. Further, we provide an alternative algebraic proof for the above result. 
\end{abstract}

\maketitle

\section{Introduction}\label{intro}
Many geometric structures on an ordinary manifold may be rephrased
as the data of a Poisson-self-commuting function on an associated super symplectic manifold
(with a refinement in the grading---a \emph{graded}\/ symplectic manifold).
This is the case, e.g., of Poisson and Courant structures.

The problem of reduction of such structures may be then equivalently rephrased in the associated
super version. The most general reduction in symplectic geometry is that of presymplectic
submanifolds. In the case at hand, besides generalizing this result to the super case,
one has to find conditions for the associated function
to descend and to be still self-commuting on the quotient. In the special case of coisotropic
submanifolds the first condition implies the second, but in general this is not the case.
Some rather general sufficient, but not necessary, conditions may be worked out in a 
reduction-by-stages framework.

Once these results have been obtained, they can be translated back into the ordinary diffeogeometric
language. Namely, one gets a general reduction theory for Poisson manifolds in terms of vector
bundles on submanifolds satisfying certain conditions. Once the result is known, it can also
be proved directly without reference to supergeometry, see Section~\ref{algproof}

All this generalizes the known reduction procedures in Poisson geometry, in particular the celebrated
Marsden--Ratiu reduction \cite{MR}.

In this note we concentrate on the main results of this approach referring to \cite{CZ}
for more details and complete proofs. The classical proof of Section~\ref{algproof} is new.
The case of reduction of Courant algebroids and of
generalized complex structures will be treated in \cite{BCMZ}.

\section{Poisson manifolds}\label{gerst}
\begin{definition}\label{poism}
$M$ is a \emph{Poisson manifold}  if
$C^{\infty}(M)$ is endowed with a Lie bracket $\{\bullet,\bullet\}$ satisfying 
\begin{equation}\label{leib}
\{f,gh\}=\{f,g\}h+g\{f,h\}.
\end{equation}
\end{definition}

Let $M$ be a smooth manifold. The Lie bracket of vector fields on $M$ extends to a bracket -- called \emph{Schouten bracket} -- on 
all multivector fields on $M$. The Schouten bracket, together with the wedge product, endows the set of multivector fields $A:=\Gamma(\wedge^{\bullet} TM)$ 
  with the structure of a \emph{Gerstenhaber algebra} (also called \emph{graded Poisson algebra of degree $1$}).
This means that $A$ is a graded commutative algebra, that 
$A[1]$ (defined by $A[1]_i=A_{i+1}$) is a graded Lie algebra (see \cite{Wa}), and that the two structures are compatible in the sense that the adjoint action of a homogeneous element $X \in A$ is a graded derivation of degree $deg(X)-1$:
\begin{equation}\label{sLeibniz}
[X,Y\wedge Z]=[X,Y]\wedge Z +(-1)^{(deg(X)-1)\cdot deg(Y)}Y\wedge [X,Z].
\end{equation}
The Schouten bracket between a  vector field and a function  is $[X,f]=X(f)$, and the 
Schouten bracket of two vector fields is the usual Lie bracket. This determines the Schouten bracket on the whole of $\Gamma(\wedge^{\bullet} TM)$ by virtue of 
 \eqref{sLeibniz}. 
 
The Gerstenhaber algebra of multivector fields is relevant for us because it allows us
to describe a Poisson manifold $M$ as a manifold with a bivector field $\pi\in \Gamma(\wedge^2 TM)$ satisfying the Schouten-bracket relation $[\pi,\pi]=0$.
The connection to Def. \ref{poism} is established as follows:  the bracket $\{\bullet,\bullet\}$ is encoded by a bivector field $\pi$ due to \eqref{leib}, the correspondence being $\{f,g\}=\pi(df,dg)$. The fact that $\{\bullet,\bullet\}$  satisfies the Jacobi identity is equivalent to $[\pi,\pi]=0$.

\section{Graded manifolds and the problem}\label{gmprobl}

Ordinary manifolds are modeled on open subsets of $\RR^n$. We start describing the local model for a \emph{graded} manifold. 
\begin{definition}
Let $U\subset \RR^n$ open subset and
$V=\oplus_{i\neq 0}V_i$ a $\ZZ$-graded vector space.\\
 The \emph{local model for a graded manifold} consists of the pair
\begin{itemize}
\item $U$ (the ``body'')\\
\item $C^{\infty}(U)\otimes S^{\bullet}(V^*)$ (the graded commutative algebra of ``functions'').
\end{itemize}
\end{definition}
Notice that here $S^{\bullet}(V^*)$ denotes the \emph{graded} symmetric algebra over $V^*$, so its homogeneous elements anticommute if they both have odd degree. 
 
\begin{definition}
A graded manifold consists of a pair as follows:
\begin{itemize}
\item a topological space $M$ (the ``body'')\\
\item a sheaf $\cO_M$ over $M$ of graded commutative algebras, locally isomorphic to the
above local model (the sheaf of ``functions'').
\end{itemize}
\end{definition}
We use the notation $C(\cM)$ for $\cO_M(M)$, the space of ``functions on $\cM$''.

A   graded vector bundle
$E=\oplus_{i\neq 0}E_{i} \rightarrow M$
can be viewed as a  graded manifold with body $M$ and functions $\Gamma(S^{\bullet}E^*)$. Recall that $E[1]$ is defined by $(E[1])_i=E_{i+1}$.

\begin{ep}
$T^*[1]M $ is a 
 graded manifold with body $M$ and functions 
$$\Gamma(S^{\bullet}( T[-1]M))=\Gamma(\wedge^{\bullet} TM)=\{\text{multivector fields on }M\}.$$
Explicitly, we can choose
 coordinates $x_j$ on $M$, giving rise  coordinates $p_j$ on the  fibers of $T^*M$; assigning degree $1$ to them we obtain    coordinates 
 $\xi_j$ of fibers of $T^*[1]M$. In these coordinates, a vector field $a_i(x)\partial_{x_i}$ corresponds to the degree $1$ function  $a_i(x)\xi_i$ on $T^*[1]M$. 
\end{ep}

Exactly as usual cotangent bundles,  
$T^*[1]M$ has a symplectic form $\omega=dx_j\wedge d\xi_j$, which gives rise
 a Poisson bracket of degree $-1$ on $C(T^*[1]M)$ determined by $\{\xi_j, x_k\}=\delta_{jk}$,
$\{\xi_j, \xi_k\}=0$, $\{x_j, x_k\}=0$ and the Leibniz rule.  
But this is just the Schouten bracket on multivector fields!
Hence we see that the degree $1$ graded Poisson algebra structure on the functions on
$T^*[1]M$ coincides with the one defined in Section \ref{gerst}.
Summarizing, we obtain the following bijective correspondences (this is Prop. 4.1 of
\cite{Dima}):
\begin{proposition}\label{equivdef}
Poisson bracket $\{\bullet,\bullet\}$  on $M$ $\leftrightarrow$\\
bivector field $\pi\in \Gamma(\wedge^2TM)$ satisfying $[\pi,\pi]=0$ $\leftrightarrow$\\
degree 2 function $\cS$ on $T^*[1]M$ satisfying $\{\cS,\cS\}=0$.
\end{proposition}
 
Hence a Poisson structure on a manifold $M$ can  equivalently be regarded as a very simple kind of structure -- indeed, just a function -- on $T^*[1]M$.\\
 
In the rest of this note we want to consider the following reduction problem: 
\begin{center}
\fbox{
\parbox[c]{12.6cm}{\begin{center}
  Let $M$ be a Poisson manifold. Specify geometric data on $M$ out of which one can construct canonically a new Poisson manifold. 
\end{center}
}}
 \end{center}


In virtue of Prop. \ref{equivdef} the problem becomes: specify geometric data on the pair $(T^*[1]M,\cS)$ which allow us to construct canonically
a new degree 1 symplectic manifold -- see Remark \ref{gradsym}; it will be again of the from $T^*[1]X$ for some manifold $X$ -- and a self-commuting degree 2 function on it.

A common way in ordinary symplectic geometry to construct new symplectic manifolds is to take  a  submanifold $N$ which is presymplectic (i.e.  $ker(\iota^*\omega)= TN^{\omega}\cap TN$   has constant rank; the special case where $TN^{\omega}\subset TN$ is called coisotropic) and to consider the quotient $N/ ker(\iota^*\omega)$, which is  automatically symplectic if smooth. 
 
This suggests to consider presymplectic submanifolds $\cC$ of $T^*[1]M$ so that  
the function $\cS$ descends to the quotient of $\cC$ by its characteristic distribution and is self-commuting there.
In the next sections we will carry this out, and in Thm \ref{A2} we will give an answer to the above reduction problem.

\begin{remark}\label{gradsym}
A \emph{degree $n$ graded symplectic manifold} is a $\ZZ_{\ge 0}$-graded manifold 
endowed with  a non-degenerate, closed 2-form whose corresponding Poisson bracket has degree $-n$. The degrees of coordinates on the graded manifold lie between $0$ and $n$
(Lemma 2.4 of \cite{Dima}). A choice of degree $n+1$ self-commuting function determines
a geometric structure by the so-called derived bracket construction. 

For $n=1$ this geometric structure is the one of a Poisson manifold (Prop. \ref{equivdef}). All degree $1$ symplectic manifolds are  of the form $T^*[1]X$ for some $X$ (Prop. 3.1 of \cite{Dima}).

 For $n=2$ the geometric structure is a so-called Courant algebroid \cite{Dima}. Some of the constructions for the case $n=2$ carried out in this note for   $n=1$ are considered in \cite{BCMZ}.
\end{remark}

\section{Graded submanifolds}
Degree $1$ graded   manifolds $\cM$ are always  of the form $W[1]$ where $W\rightarrow M$ is a vector bundle. 
 
We define graded submanifolds of $\cM$ in terms of  coordinate functions on $\cM$. To this aim recall that
functions  $x_i$ of degree zero $(i\le dim(M))$ and $\xi_j$ of degree one $(j\le rk(W))$ defined over an open subset $U\subset M$  
are called  \emph{coordinates} if 
  the $x_i$ are usual coordinates on $U$ and 
 there is an isomorphism of graded commutative algebras from $\cO_M(U)$ to the local model $C^{\infty}(U)\otimes S^{\bullet}(V^*)$ so  that under this isomorphism the $\xi_j$ correspond a basis of $V^*$. Here    $V$ is a vector space concentrated in degree $-1$ with $dim(V)=rk(W)$.

\begin{definition}\label{def:subm} A \emph{graded submanifold} $\cC$ of $\cM$ is given by a homogeneous graded
ideal $\cI \subset C(\cM):=\cO_M(M)$ satisfying the following ``smoothness'' property.
In a neighborhood $U$ of any point $x\in M$ satisfying $\cI_0(x)=0$ there exist coordinates
 $x_i$  and $\xi_j$   
so that $\cI(U)$ is generated by $x_{dim_0(\cC)+1},\cdots,x_{dim(M)}$ and 
$\xi_{dim_1(\cC)+1},\cdots,\xi_{rk(W)}$.  
Further, we require that the vanishing set of $\cI_0$ be closed in $M$.
The integers  $dim_i(\cC)$ are called the \emph{dimensions} of $\cC$ in degree $i$ ($i=0,1$).
\end{definition}
In concrete terms, we have
 $\cI_0=Z(C)$ for some closed submanifold of $M$ and 
 $\cI_1=\tilde{\Gamma}(E)$ for   some
vector subbundle $E\rightarrow C$ of $W^*\rightarrow M$.
 Here and in the sequel we use the notation $\tilde{\Gamma}(\bullet)$
to denote sections of a vector bundle   which restrict to
sections of the subbundle $\bullet$, so $\tilde{\Gamma}(E)=\{X\in
{\Gamma}(W^*):X|_C\subset E\}$. 

Since $C(\cM)/\cI$ is canonically isomorphic to $C(E^{\circ}[1])$ we write $\cC=E^{\circ}[1]$. Here $E^{\circ}\subset W|_C$ denotes the annihilator of $E$.

\section{Coisotropic submanifolds}\label{sec:coiso}

In the previous section we saw that submanifolds of $\cM:=T^*[1]M$ are of the form $E^{\circ}[1]$ for some  
vector subbundle $E\rightarrow C$ of $TM\rightarrow M$.
  Denote by $\cI$ the ideal defining $\cC$ and by
$\cN(\cI)$ the Poisson normalizer of $\cI$, i.e. the set of
functions $\phi \in C(\cM)$ satisfying $\{\phi, \cI\}\subset \cI$.
One computes\begin{equation}\label{n0}
\cN(\cI)_0=\{f\in C^{\infty}(M):df|_C\subset
 E^{\circ}\}=:C^{\infty}_E(M),
\end{equation}
\begin{equation}\label{n1}
\cN(\cI)_1=\{X\in \tilde{\Gamma}(TC):[X,\tilde{\Gamma}(E)]\subset \tilde{\Gamma}(E)\}. 
\end{equation}


\begin{definition}
The submanifold $\cC$ is \emph{coisotropic} if $\{\cI,\cI\}\subset \cI$ (i.e. $\cI\subset \cN(\cI)$).\end{definition}
 By degree reasons
$\{\cI_0,\cI_0\}$ always vanishes. If  $X\in
\cI_1=\tilde{\Gamma}(E)$  and $f\in \cI_0=Z(C)$ we have
$\{f,X\}=-X(f)$. So $\{\cI_0,\cI_1\}\subset \cI_0$ is equivalent to
$E\subset TC$. If  $X,Y\in \cI_1$ then $\{X,Y\}=[X,Y]$, so
$\{\cI_1,\cI_1\}\subset \cI_1$ is equivalent to the involutivity of
the distribution $E$ on $C$.

In this case, since by construction $\cI$ is a Poisson ideal in the Poisson algebra $\cN(\cI)$, the Poisson bracket descends making $\cN(\cI)/\cI$ into  a graded Poisson algebra.
 In degree $0$ by eq. \eqref{n0} it consists of the $E$-invariant functions on $C$, so let us assume that the quotient $\un{C}$ of $C$ by the foliation integrating $E$ be a smooth manifold (so that $C\rightarrow \un{C}$ is a submersion).
In degree $1$ by eq. \eqref{n1} $\cN(\cI)/\cI$  consists of vector fields on $C$ which are
projectable w.r.t. the projection $C \rightarrow \un{C}$, modulo
vector fields lying in the kernel of the projection. In other words
$(\cN(\cI)/\cI)_1$ is isomorphic to the space of vector fields on
$\un{C}$. We conclude that $\cN(\cI)/\cI$ is the graded Poisson
algebra on a graded symplectic manifold if{f} $\un{C}$ is smooth, and
in that case it is the Poisson algebra of functions on
$T^*[1]\un{C}$.

Further, the function $\cS$ induces a function $\un{\cS}$ on $T^*[1]\un{C}$ if{f} $\cS\in \cN(\cI)$. In that case, by the way we defined the bracket on $\cN(\cI)/\cI$, it is clear that $\un{\cS}$ commutes with itself. Hence we obtain a reduced Poisson structure on $\un{C}$.
We spell out what it means for $\cS$ to lie in $\cN(\cI)$. Since for any function $f$ on $M$ we have  $\{\cS,f\}=[\pi,f]=\sharp df$, $\{\cS,\cI_0\}\subset \cI_1$
is equivalent to $\sharp N^*C\subset E$. Here $\sharp: T^*M \rightarrow TM$
denotes contraction with the bivector $\pi$ and $N^*C:=\{\xi \in T^*M|_C: \langle \xi, TC \rangle =0\}$. Notice that in particular $C$ is a coisotropic submanifold of $M$. Further, for any vector field $X$ on $M$,
$\{\cS,X\}=[\pi,X]=-\cL_X \pi$, so
  $\{\cS,\cI_1\}\subset \cI_2$ is equivalent to $(\cL_X \pi)|_C \in \Gamma(E\wedge TM|_C)$ for any $X\in \tilde{\Gamma}(E)$, which using eq. \eqref{Xfg} below is equivalent to $C^{\infty}_E(M)$ being closed under the Poisson bracket of $M$.

We summarize:
\begin{proposition}\label{coisocase}
A coisotropic submanifold $\cC$ of $T^*[1]M$ corresponds to a
submanifold $C$ of $M$ endowed with an integrable distribution $E$.
The coisotropic quotient of $\cC$ is smooth if{f} $\un{C}=C/E$ is
smooth, and in that case the coisotropic quotient is canonically
symplectomorphic to $T^*[1]\un{C}$. The function $\cS$ on $\cM$
descends to a degree 2 self-commuting function on $T^*[1]\un{C}$
(which therefore corresponds to a Poisson structure on $\un{C}$)
if{f} $\sharp N^*C\subset E$ and $C^{\infty}_E(M)$ is closed under
the Poisson bracket.
\end{proposition}

The Poisson-reduction result obtained from the above proposition is quite trivial.
In order to obtain more interesting results we have to allow $\cC$ to be not just a coisotropic submanifold, but actually a presymplectic
 submanifold of $T^*[1]M$.

\section{Presymplectic submanifolds}\label{sec:presympl}

We consider again a submanifold $\cC=E^{\circ}[1]$ of $\cM:=T^*[1]M$, 
 and denote by $\cI$ its vanishing ideal.
To define presymplectic submanifolds we need the following 
\begin{definition}
Let $A_{ij}$ be a matrix with entries in $C(\cC):=C(\cM)/\cI$.\\
$A$ has \emph{constant rank along $\cC$} if{f}, switching rows and adding  $C(\cC)$-multiples of a row to another row, the matrix $A$ can be  brought to the form 
$\left(\begin{smallmatrix} \star \\0\end{smallmatrix}\right)$
where the degree zero part of the  rows of $\;\star\;$ are linearly independent at every point of the body of $\cC$.
\end{definition}

\begin{definition}
A submanifold $\cC$ is \emph{presymplectic} if{f} 
$\cI$ is generated by homogeneous functions $\phi_i$ for which the matrix $\;\;\{\phi_i,\phi_j\} \text{ mod }\cI\;\;$ has constant rank along $\cC$.
\end{definition}

Translating in terms of classical geometry we obtain

\begin{lemma}\label{pres}
  $\cC$ is a graded presymplectic submanifold if{f} $TC\cap E$ is a constant rank, involutive distribution on $C$ .\\
 The quotient of $\cC$ by its characteristic distribution, defined as  $\cN(\cI)/\cN(\cI)\cap \cI$, is smooth
if{f} the quotient $\un{C}:=C/(TC\cap E)$ is smooth. In this case
it is isomorphic to $T^*[1]\un{C}$ as a graded symplectic manifold.
\end{lemma}

Now we address the issue of when the function $\cS$ induces a
function $\un{\cS}$ on the quotient $\un{\cC}:=T^*[1]\un{C}$. $\cS$
descends if{f} its image under the map $C(\cM)\rightarrow C(\cM)/\cI$
lies in
 $\cN(\cI)/\cN(\cI)\cap \cI$, i.e. if{f} $\cS$ lies in $\cN(\cI)+ \cI$.

When $\cS$ descends,
 $\un{\cS}$ might not commute with itself. The reason is that  the Poisson bracket on
 $\cN(\cI)/\cN(\cI)\cap \cI$ is computed lifting to elements of $\cN(\cI)$ (and not to arbitrary elements of $\cN(\cI)+\cI$).
 
It is clear that if $\cS$ lies in $\cN(\cI)$ then the induced function on $\un{\cC}$ still commutes with itself. It turns out that it suffices to require that  
\begin{equation}\label{halfcond}
\{\cS,\cI_0\}\subset \cI_1
\end{equation}
 (or equivalently
$\sharp TC^{\circ}\subset E$); this conditions leads to the statement
of Prop. 5.17 of \cite{CZ} and  Prop. 4.1 of \cite{FZ}, which is a mild improvement of that of \cite{MR}. We do not state it here because in Thm. \ref{A2} we will state a yet better result.

\section{Reduction in stages and the theorem}\label{redstages}

To derive a condition weaker than \eqref{halfcond} we perform
reduction in stages, as follows. We imbed the presymplectic submanifold $\cC$ in a larger
 coisotropic  submanifold $\cA$ of $\cM$. We assume that
 the quotient
$\un{\cC}$ of $\cC$ by
 its characteristic distribution $T\cC\cap T\cC^{\omega}$ is smooth. Locally 
the quotient
 can be realized in two stages:
  first
 take the image $\bar{\cC}$ of $\cC$ under the projection
$\cA\rightarrow \bar{\cA}:=\cA/T\cA^{\omega}$; assuming that $T\cC
\cap T\cA^{\omega}$ has constant rank,  $\bar{\cC}$ is a
presymplectic submanifold. Then take the presymplectic quotient of
$\bar{\cC}$. It will be (locally) symplectomorphic to $\un{\cC}$. 
 Now assume that
\begin{eqnarray}
 \label{SdescCnice} &&\cS \text{ descends to }\un{\cC}\\
 \label{SnormAnice} &&\cS \text{ descends to a function }\bar{\cS} \text{ on
 }\bar{\cA}\\
  \label{ShalfnormCnice} &&\bar{\cS} \text{ satisfies condition } \eqref{halfcond}, \text{ i.e. } \{\bar{\cS},(\cI_{ \bar{\cC}})_0\}\subset (\cI_{ \bar{\cC}})_1.
\end{eqnarray}
Then the reasoning of the previous section implies that the function $\un{\cS}$ on
$\un{\cC}$ commutes with itself. Since we are ultimately  interested in a
quotient of $\cC$, it is clear that  
condition \eqref{SnormAnice} can be weakened. 

The geometric procedure described above is carried out in algebraic terms in \cite{CZ}. Writing $\cA=D^{\circ}[1]$ for a subbundle $D\rightarrow A$ of $TM$, we obtain
  
 \begin{theorem}\label{A2}
Let $C$ be a submanifold of the Poisson manifold $M$ and $E\subset TM|_C$ a subbundle  such that $F:=TC\cap E$ is a constant rank, involutive distribution on $C$.

Let $D|_C$ be a subbundle of $TM|_C$ with $F\subset D|_C\subset E$
and
\begin{equation}\label{etcd}
\sharp E^{\circ}\subset TC+D|_C.
\end{equation}
Let $A$ be a submanifold containing $C$ such that $TA|_C=TC+D|_C$, and
 assume that  $D|_C$ can be
extended to an integrable distribution $D$ on $A$ such that
\begin{equation}\label{Liedercon}
(\cL_{X_i}\Pi)|_C \subset E\wedge TM|_C
\end{equation}
where $\{X_i\}$ is an extension from $A$ to $M$ of a (local) frame of
sections of $D$.

Then $\un{C}:=C/F$ inherits a Poisson manifold structure.
\end{theorem}

\section{An algebraic proof}\label{algproof}

Without the graded geometric interpretation it would have been hard to derive  Thm. \ref{A2}. Once the statement is known, however, it is easy to give
 an alternative algebraic proof. We will do so in this section. 
 
The following algebraic statement 
 (compare also to Prop. A.1 in \cite{FZ})  reduces to an obvious one 
 when $\cB=\cD$, for in that case
 $\cB\cap \cI$ a Poisson ideal in the 
 Poisson subalgebra $\cB$.

\begin{proposition}\label{alg}
Let $\cP$ be a Poisson algebra, $\cB \subset \cD$ multiplicative
subalgebras of $\cP$ and $\cI$ a multiplicative ideal of $\cP$.
Assume that
\begin{equation}\label{braalg}
 \{\cB,\cB\}\subset \cD  \cap (\cI + \cB)
\end{equation}
and
\begin{equation}\label{shaalg}
    \{\cB,\cI\cap \cD\}\subset \cI.
\end{equation}
Then there is an induced Poisson algebra structure on
$\frac{\cB}{\cB\cap \cI}$, whose bracket is determined by the
commutative diagram
\[
\xymatrix{
 \cB \times \cB \ar[d] \ar[r]^{\{\cdot,\cdot\}}& \cD  \cap (\cI + \cB)\ar[d]
\\
\frac{\cB}{\cB\cap \cI} \times \frac{\cB}{\cB\cap \cI}  \ar[r]&
 \frac{\cB}{\cB\cap \cI}\\
 }.
\]
\end{proposition}

\begin{proof}
The above diagram is well-defined because of \eqref{braalg} and $\{\cB,\cI\cap \cB\}\subset \cI$ (which holds by \eqref{shaalg}). The induced bilinear operation on $\frac{\cB}{\cB\cap \cI}$ satisfies the Leibniz rule \eqref{leib} because the Poisson bracket on $\cP$ does. To check the Jacobi identity consider $f,g,h\in \frac{\cB}{\cB\cap \cI}$ and lifts
$\tilde{f}$,$\tilde{g}$,$\tilde{h}$ and $\widetilde{\{g,h\}}$ to elements of $\cB$. Since 
$\widetilde{\{g,h\}}$ and $\{\tilde{g},\tilde{h}\}$ are  lifts of the same element, using
again \eqref{braalg} we see that their difference $\Delta$ lies in 
$\cI\cap[\cD  \cap (\cI + \cB)]=\cI\cap \cD$. Hence
$$\{f,\{g,h\}\}=\{\tilde{f},\widetilde{\{g,h\}}\}\;\; mod \;\;\cI=
\{\tilde{f}, \{\tilde{g},\tilde{h}\}\}+  \{\tilde{f},\Delta \}\;\; mod \;\;\cI$$ 
where  $\{\tilde{f},\Delta \}$ lies in $\cI$ by \eqref{shaalg}. Taking the cyclic sum over $f,g,h$ we see that the Jacobi identity on $\frac{\cB}{\cB\cap \cI}$ follows from the one on $\cP$.
\end{proof}

\begin{lemma}
Theorem \ref{A2} follows from Prop. \ref{alg} setting  $\cP=C^{\infty}(M)$,
$\cD=C^{\infty}(M)_{D|_C}$, $\cB=C^{\infty}_E(M) \cap C^{\infty}_D(M)$
 and $\cI=\{f\in C^{\infty}(M): f|_C=0\}$. Here we use the notation introduced for $C^{\infty}_E(M)$  at the beginning of  Section  \ref{sec:coiso}. 
\end{lemma}
\begin{proof}
Condition \eqref{shaalg} is satisfied because of requirement \eqref{etcd}.
Now we check condition \eqref{braalg} in two steps. 

First, using \eqref{etcd}
(i.e. $\sharp E^{\circ}\subset TA|_C$), 
requirement \eqref{Liedercon} is equivalent to $\{\cB,\cB\}\subset \cD$.
To see this apply to $f,g\in \cB$ and $X\in \tilde{\Gamma}(D)$ the identity
\begin{equation}\label{Xfg} X\{f,g\}=
(\cL_X
\Pi)(df,dg)+\Pi(d(Xf),dg)+\Pi(df,d(Xg)).\end{equation}

Second, we have $ \{\cB,\cB\}\subset \cI + \cB$. Indeed the bracket of two elements of $\cB$ annihilates $D|_C$ by the above, so its restriction to $C$ annihilates $D|_C\cap TC=F$.
So it suffices to show: any function $h$ in $C^{\infty}_F(C)$ can be extended to a function in $\cB$. Using fact that locally  $C/F$ embeds naturally into $A/D$  
we can extend $h$ to a function in $C^{\infty}_D(A)$. Choosing a complement of $E\cap TA|_C=D|_C$ in $E$, we can extend to an element of $\cB$. 

Hence the assumptions of  Prop. \ref{alg} are satisfied, and therefore
$\frac{\cB}{\cB\cap \cI}$ is a Poisson algebra. It is clear that 
$\frac{\cB}{\cB\cap \cI} \subset C^{\infty}_F(C)$. Equality holds because, as shown above,
any function  in $C^{\infty}_F(C)$ can be extended to a function in $\cB$.
\end{proof}
   
\noindent\textbf{Acknowledgments: }
M.Z. thanks the organizers of the conference ``Special
metrics and supersymmetry" in Bilbao (2008) for the invitation and for local support, and C.R.M. Barcelona for partial financial support.  
A.S.C. is grateful to C.R.M. Barcelona
for hospitality.
This work has been
partially supported by SNF Grant 20-113439, by
the European Union through the FP6 Marie Curie RTN ENIGMA (contract
number MRTN-CT-2004-5652), and by the European Science Foundation
through the MISGAM program.


\end{document}